\documentclass{amsart}

\usepackage{amssymb}
\usepackage{verbatim}
\usepackage{hyperref}
\usepackage[english]{babel}
\usepackage{wasysym}
\usepackage{graphics}

\newtheorem{theorem}{Theorem}[section]
\newtheorem{lemma}[theorem]{Lemma}
\newtheorem{corollary}[theorem]{Corollary}

\theoremstyle{definition}
\newtheorem{definition}[theorem]{Definition}
\newtheorem{example}[theorem]{Example}

\newtheorem{problem}[theorem]{Problem}
\newtheorem{proposition}[theorem]{Proposition}

\theoremstyle{remark}

\numberwithin{equation}{section}

\newcommand{\ff}{\mathbb{F}}
\newcommand{\gl}{\operatorname{GL}}
\newcommand{\sym}{\operatorname{Sym}}
\newcommand{\symnull}{\sym_0}
\newcommand{\psl}{\operatorname{PSL}}
\newcommand{\pgl}{\operatorname{PGL}}
\newcommand{\aff}{\operatorname{AGL}}
\newcommand{\symp}{\operatorname{Sp}}
\newcommand{\symptr}{\Delta}
\newcommand{\aut}{\operatorname{Aut}}
\newcommand{\negyz}{\diamond}
\newcommand{\otsz}{\pentagon}
\newcommand{\hszog}{\triangledown}

\begin{document}

\title[Supergroups of the infinite linear groups]{Permutation groups containing infinite linear groups and
  reducts of infinite dimensional linear spaces over the two element field}

\author[B. Bodor]{Bertalan Bodor }
\address[1,2,3]{E\"{o}tv\"{o}s Lor\'{a}nd University, 
          Department of Algebra and Number Theory,
         1117 Budapest,  P\'{a}zm\'{a}ny P\'{e}ter s\'{e}t\'{a}ny 1/c, Hungary}
\email{bodorb@cs.elte.hu}

\author[K. Kalina]{Kende Kalina }
\email{kkalina@cs.elte.hu}

\author[Cs. Szab\'{o}]{Csaba Szab\'{o}}
\thanks{The research was supported by the Hungarian OTKA K109185 grant.}
\email{csaba@cs.elte.hu}



\date{2014.07.28}

\begin{abstract}
Let $\mathbb{F}_2^\omega$ denote the countably infinite dimensional
vector space over the two element field and $\operatorname{GL}(\omega, 2)$ its 
automorphism group. Moreover, let $\operatorname{Sym}(\mathbb{F}_2^\omega)$ denote the
symmetric group acting on the elements of  $\mathbb{F}_2^\omega$. It is shown
that there are exactly four closed subgroups, $G$, such that
$\operatorname{GL}(\omega, 2)\leq G\leq \operatorname{Sym}(\mathbb{F}_2^\omega)$. As $\mathbb{F}_2^\omega$ is
an $\omega$-categorical (and homogeneous) structure, these groups
correspond to the first order definable reducts of
$\mathbb{F}_2^\omega$. These reducts are also analyzed. In the last section the closed groups containing the infinite symplectic group
$\operatorname{Sp}(\omega, 2)$ are classified.
\end{abstract}

\maketitle

\section{Introduction}

Let $\ff_2^\omega$ denote the countably infinite dimensional
vector space over the two element field and $\gl(\omega, 2)$ its 
automorphism group. Moreover, let $\sym(\ff_2^\omega)$ denote the
symmetric group acting on the elements of  $\ff_2^\omega$. In this
paper the closed subgroups of $\sym(\ff_2^\omega)$ containing 
$\gl(\omega, 2)$ are investigated. In the last section this results are extended to closed groups containing the infinite symplectic group
$\operatorname{Sp}(\omega, 2)$. Closure of subgroups means being closed in the topology of pointwise convergence. For an infinite set $\Omega$ a
subgroup  $H\leq \sym(\Omega)$ is closed if the following condition
holds: for every
$\pi \in \sym(\Omega)$ and every finite subset $S\subset \Omega$ if there
is some $\sigma_S \in H$ such that $\sigma_S|_S=\pi|_S$ then  $\pi
\in H$. 

For
finite dimensional projective spaces it is almost independently shown 
in \cite{PB}, \cite{KMcD} and \cite{Pog}  that if $\psl(n,q)$, the
special linear group is contained
in a subgroup $G$ of the symmetric group acting on the points of the
projective space, then $G$ is either the alternating or the full
symmetric group,  or $G$ is contained in the twisted projective linear group.
If we investigate  the infinite version of this theorem, we have to
consider the following:
 The alternating group has an infinite counterpart - the group of finite
support even permutations – but it is not closed. 
As $A_n\geq S_{n-2}$ for every finite integer $n$, the even
permutations on a set of 
size $\omega $ generate all permutations on every finite subset: for
every finite $S\subset \omega$ and $\pi \in \sym (S)$ there is
some $\mu\in \sym(\omega)$ such that $\mu|_{S}=\pi$. Hence the closure
of the subgroup generated by the even permutations in $\sym(\Omega)$
is $\sym(\Omega)$ itself. Similarly, ''$\psl(\omega, q)$'' does not
make sense, also, the closure of the subgroup of $\sym(P(\omega,q))$
generated by the matrices of 
determinant 1 is $\pgl(\omega,q) $.

In this  paper we consider the
countably infinite dimensional vector space, $\ff_2^\omega$, and prove
that if the automorphism group of the
countably infinite
dimensional  vector space, $\gl(\omega,2)$ 
is  contained in a
closed (nontrivial)  subgroup $G$ of the symmetric group
$\sym(\ff_2^\omega)$, then $G$ is either the affine group or
$\sym_0(\ff_2^\omega)$, the stabilizer of $0$.
As for the two element field the linear and projective linear groups
are essentially the same, as a corollary,  we obtain that the
projective linear group is a maximal closed group of $\sym(\ff_2^\omega)$.
We also obtain that the closed groups containing the infinite symplectic group
$\operatorname{Sp}(\omega, 2)$ are exactly the groups containing $\gl(\omega,2)$,
the group $\operatorname{Sp}(\omega, 2)$ itself and the group generated by $\operatorname{Sp}(\omega, 2)$ and the group of translations.

One would expect similar characterisation for all finite fields of prime size.  We show that contrary to the expectations there are several
more such closed subgroups for primes greater than 2.
 We exhibit $ d(p-1)$ many closed supergroups  of $\gl(\omega,
 p)$  for every prime $p>2$, 
where $d(n)$ denotes the number of the divisors of $n$.  This shows
that for fields of larger   size the problem is far more
complicated.  

The vector space $\ff_2^\omega$ is homogeneous in the sense that every
(partial) automorphism between two finite substructures can be extended
to an automorphism of $\ff_2^\omega$. For more details on the model
theoretical background consult~\cite{Hodges}.
Closed supergroups of automorphism groups of homogeneous structures
have special importance.
Closed supergroups of the automorphism group
are in a Galois-connection with the first order definable reducts of
the structure. A relational structure is first order definable from an
other one if it has the same underlying set and can be defined by first order sentences. First order
interdefinability is an equivalence relation, and the classes dually
correspond to the closed subgroups of the full symmetric group
containing the automorphism group of the structure.

For $\omega$-categorical structures this is a bijection between the equivalence classes of reducts and the closed groups containing
the automorphism group of the structure. Since the structures investigated in this paper are $\omega$-categorical the description of the reducts up to
first order interdefinability and the description of the closed groups containing the automorphism group of a given structure is equivalent.

The $\omega$-categoricity can be checked using the following theorem \cite{Hodges}:

\begin{theorem}[Engeler, Ryll-Nardzewski, Svenonius]\label{ERNS}
A countable structure is $\omega$-categorical if and only if its automorphism group is oligomorphic i. e. for every $n$ the automorphism group acting on
the ordered $n$-tuples of the elements of the structure has finitely many orbit.
\end{theorem}

Until recently
only sporadic examples of first order definable reducts were known. In
\cite{cam} the reducts of the dense linear order are determined. For
the random graph and random hypergraphs Thomas has determined their
reducts \cite{thomas1}, \cite{thomas2}. Both the random graph and the
dense linear order have 5 reducts. In \cite{JUZI} it is shown that the
``pointed'' linear order has 116 reducts. Thomas has conjectured that
any homogeneous structure on a finite relational language has finitely many
reducts.  Later in \cite{BPT} and \cite{BP11a} a general technique was
introduced to
investigate first order definable reducts of homogeneous structures on
a finite language. Then several structures
were analyzed from this aspect: the pointed Henson-graphs \cite{pong1},
the random poset \cite{pppsz}, \cite{ppppsz},  equality \cite{BCP} and
the random graph revisited \cite{BP}.

It is argued in \cite{Mac11} that 
the homogeneous vector spaces and the affine spaces cannot be defined
by any finite relational language. Hence our work is a first attempt to classify
first order definable reducts of homogeneous structures on an infinite
relational language. 

We aimed to have our proofs as elementary as possible. All
finitary versions of our theorem and recent results on homogeneous
structures use rather difficult techniques. Although it is tempting
to refer to those results, we kept this paper self-contained in this
sense.

\section{Supergoups of $\operatorname{GL}{(\mathbb{F}_2^\omega)}$ }\label{group}

Let us introduce some notation, first. Let the group $G$ act on the set
$\Omega$. For an element  $s\in \Omega$ let $G_s$ denote the 
stabilizer of $s$ in $G$ and for $s_1,s_2,\dots, s_n$  let
$G_{(s_1,s_2,\dots, s_n )}$ denote the elementwise stabilizer of the
elements.  Let $Aff(\omega,2)$ denote the affine 
space 
obtained  from $\ff_2^\omega$. Let $\aff$ denote the automorphism
group of $Aff(\omega,2)$.  The group $\aff$ is generated by $\gl(\omega,2)$
and the translations. Let $a\in \ff_2^\omega$. The translation by $a$,
denoted by $t_a$ is defined by $v^{t_a}=v+a$. Note that
$\gl(\omega,2)$ along with any translation generates $\aff$.
 We shall
use the elementary facts about linear algebra (as the extendability of
a map from a basis to a linear transformation)  without any reference.

We summarize the necessary information about $\aff$.

\begin{lemma}\label{Rkarakterizal}
The group $\aff$ contains exactly those $f$ permutations of $\sym
(\ff_2^\omega)$ that preserve the ternary addition:
for all $ a, b, c \in
\ff_2^\omega $ the equality $(a+b+c)^f=a^f+b^f+c^f  $ holds, or
alternatively if
$a+b+c+d=0$ we have $a^f+b^f+c^f+d^f=0$. In particular $\aff$ is closed.
\end{lemma}

\begin{proof}

Let $f$ be such that
for all $ a, b, c \in
\ff_2^\omega $ the equality $(a+b+c)^f=a^f+b^f+c^f  $ hold.
Let the map $g $ be defined by
$x^g=x^f+0^f$. 
Then $g\in \gl(\omega, 2)$, because $0^g=0^f+0^f=0$ and $(a+b)^g=(a+b+0)^f+0^f=a^f+b^f+0^f+0^f=a^g+b^g$.
Since $x^f=x^g+0^f$ the permutation $f$ is contained in $\aff$, it is the composition of
the vector space automorphism $g$ and the translation $t_{0^f}$. 

The other direction is obvious as both the elements of
$\gl(\omega,2)$ and the translations preserve the ternary addition.
\end{proof}

The following lemma will be applied several times.

\begin{lemma}\label{kislemma}

Let us assume that   $G \leq \sym(\ff_2^\omega)_0$ and $G$ acts
$n$-transitively on $\ff_2^\omega \setminus \{ 0 \}$.

Moreover, let us assume that for every $(x_1, x_2 \ldots x_k) \in \ff_2^\omega \setminus \{ 0 \}$ finite tuple of elements
and every $ y \in \ff_2^\omega \setminus \left< x_1, x_2 \ldots x_k \right>$ the element $y$ has an infinite orbit in the pointwise stabilizer
of the tuple $(x_1, x_2 \ldots x_k)$ in $G$.

Let $a_1, a_2 \ldots a_n, a_{n+1}
\in \ff_2^\omega \setminus \{ 0 \}$ be distinct elements
such that $a_{n+1} \notin \left<a_1, a_2 \ldots a_n\right>$.
Then there is an $ h \in G$ such that $a_1^h, a_2^h, \ldots, a_n^h,
a_{n+1}^h$ are linearly independent.
\end{lemma}

\begin{proof}
Since  $G$ acts $n$-transitively on $\ff_2^\omega \setminus \{ 0 \}$ we can
choose an element $g \in G$ such that $a_1^g, a_2^g,  \ldots, a_n^g$ are linearly
independent. Let $W=\left< a_1^g, a_2^g,  \ldots, a_n^g \right>$.
Now, $W^{g^{-1}}$
is a finite set, and the orbit of $a_{n+1}$ in the stabilizer of $\{
a_1, a_2 \ldots a_n \}$ is infinite, hence we can choose 
an $h \in G_{( a_1, a_2 \ldots a_n )}$ such that $a_{n+1}^h \notin
W^{g^{-1}}$. Then $a_i^{hg}=a_i^g$ for $1\leq i\leq n$ and they are
independent, moreover 
$a_{n+1}^{hg}\notin  W$, hence $a_1^{hg}, a_2^{hg},  \ldots, a_n^{hg},
a_{n+1}^{hg}$ are linearly 
independent.
Thus $hg$ satisfies the condition of the
Lemma.
\end{proof}

\begin{theorem}\label{0fix}
Let us assume that   $ \gl(\omega,2) \lneqq G \leq \sym(\ff_2^\omega)_0$ and $G$
is closed. Then
${G} =\sym(\ff_2^\omega)_0$
\end{theorem}
\begin{proof} It is enough to prove that $G$ is $n$-transitive for
every finite $n$. 

At first we show 3-transitivity. The linear
group $\gl(\omega,2)$ acts transitively on the 3-element independent
sets, hence it is enough to show that any three vectors $a,b,c\in \ff_2^\omega$
can be mapped to and independent set. The vectors $a,b,c$ are
dependent exactly if $a+b=c$. The condition  $ \gl(\omega,2) \lneqq
G$ implies that there are $a',b',c'\in \ff_2^\omega$ and $g\in G$ such that
$a'+b'=c'$ and $a'^g+b'^g\neq c'^g $. Now, consider a map $h\in \gl(\omega,2)$
mapping $a,b,c$ to  $a',b',c'$, respectively. The map $hg$ maps
$a,b,c$ to an independent set.

Now, we prove $n$-transitivity by induction. We show that every set of
$n+1$ vectors can be mapped to an independent set.
Let  $a_1, a_2, \ldots, a_n, a_{n+1}
\in \ff_2^\omega \setminus \{ 0 \}$ be dependent distinct elements. By the
$n$-transitivity we may assume that $a_n=a_1+ a_2 +\ldots+ a_{n-1}$ and
there is an $h\in G$ such that  $\{a_i^h|i=1,2,\dots,n\}$ is
a linearly
independent set. If $a_{n+1} \notin \left<a_1,
a_2, \ldots, a_n\right>$, then   by Lemma~\ref{kislemma} we are done. 
If $a_{n+1} \in \left<a_1,
a_2, \ldots, a_n\right>$, then $a_{n+1}=\sum\limits_1^{n-1} \varepsilon_i
a_i $  where at least two, but not all $\varepsilon_i$ are equal to 1. 
Indeed, assume that there is a unique $i$ such that $\varepsilon_i=1$,
then $a_{n+1}=a_i $   would hold, and if all of them were equal to 1,
then
$a_{n+1}=\sum\limits_1^{n-1} 
a_i =a_n$ would hold contradicting that the vectors are distinct.
Let $\varepsilon_j=1$ and $\varepsilon_k=0$  for some $j,k<n$. Then
there is a linear map $g$ flipping $a_j$ and $a_k$ and fixing every
$a_i$, where $i<n$ and $i\neq j,k$. Now,
$\{a_i^h|i=1,2,\dots,n\}=\{a_i^{gh}|i=1,2,\dots,n\}$ is an
independent set and $a_{n+1}^h\neq a_{n+1}^{gh} $. If any of the latter two
elements is not in $\left<a_1^h,
a_2^h, \ldots. a_n^h\right>$ then we are done by
Lemma~\ref{kislemma}. Otherwise we may assume that $
a_{n+1}^h=\sum\limits_1^n \xi_i a_i^h  $, where there is an $l$ such
that $\xi_l=0$. Now, $a_l^h\notin 
\left<a_i^h| 1\leq i\leq n+1, i\neq l\right>$
and we are done again, by Lemma~\ref{kislemma}.
\end{proof}

Now, we consider the case  when 0 is not fixed by $G$

\begin{lemma}\label{affin}
Let us assume that $G\leq \sym(\ff_2^\omega)$, and $G_0=G\cap (\sym(\ff_2^\omega))_0\leq \gl(\omega,2)$. 

Moreover, let us assume that for every $(x_1, x_2, x_3) \in \ff_2^\omega \setminus \{ 0 \}$ finite tuple of elements
and every $ y \in \ff_2^\omega \setminus \left< x_1, x_2, x_3 \right>$ the element $y$ has an infinite orbit in the pointwise stabilizer
of the tuple $(x_1, x_2, x_3)$ in $G$.
Then $G\leq \aff$.
\end{lemma}

\begin{proof}
Recall that the affine group $\aff$ is the set of
elements of $\sym \left( \ff_2^\omega \right)$ 
preserving the ternary addition, or alternatively, preserving the
4-tuples $(a,b,c,d)$ satisfying $a+b+c+d=0$.

The pointwise stabilizer of an injective $3$-tuple $(0,a,b)$ fixes the element $a+b$ because this stabilizer fixes the $0$
so it is in $\gl$. By the assumption of the lemma this stabilizer does not fix any other elements in $\ff_2^\omega\setminus \{0,a,b\}$. Because of the transitivity this implies that the pointwise stabilizer of any injective $3$-tuple $(a,b,c)$ must have exactly one
fixed point in $\ff_2^\omega\setminus \{a,b,c\}$. Let us denote this fix point by $F(a,b,c)$. By the assumption of the lemma we know that $F(a,b,c)$ must be contained in the subspace $\langle a,b,c \rangle$. In particular it implies that if $c=a+b$ and $a,b,c\neq 0$, then $F(a,b,c)=0$.

	Now, let $a,b,c\in \ff_2^\omega$ be linearly independent elements. Then $F(a,b,c)\in \langle a,b,c \rangle$. We claim that $F(a,b,c)=a+b+c$.
	Suppose not. Then $F(a,b,c)\in \{0,a+b,a+c,b+c\}$. By symmetry we can assume that $F(a,b,c)$ is either 0 or $a+b$.
	In any case $c\not\in \langle a,b,F(a,b,c) \rangle$, hence $c\neq F(a,b,F(a,b,c))$. The group $G$ is transitive, thus there exists a
	permutation $g\in G$ such that $a^g=0$. By the definition of the function $F$ we know that $F(x^g,y^g,z^g)=(F(x,y,z))^g$ holds for any
	injective $3$-tuples $(x,y,z)$. Therefore 
\begin{multline*}
F(a,b,F(a,b,c))=(F(a^g,b^g,F(a^g,b^g,c^g)))^{g^{-1}}=\\=(F(0,b^g,F(0,b^g,c^g)))^{g^{-1}}=(F(0,b^g,b^g+c^g)^{g^{-1}}=(c^g)^{g^{-1}}=c.
\end{multline*}
This is a contradiction, hence $F(a,b,c)$ must be $a+b+c$. We obtained that the group $G$ preserves the ternary addition $x+y+z$, therefore $G\leq \aff$.
\end{proof}

\begin{theorem}\label{0nofix}
Let us assume that   $ \gl(\omega,2) \lneqq G \leq \sym(\ff_2^\omega)$, and $G$
is closed  not fixing 0.  Then
$G =\aff$ or  $G=\sym(\ff_2^\omega)$
\end{theorem}

\begin{proof}
The stabilizer of the $0$ in $G$ is $\gl$ or $\symnull$. If it is $\symnull$ then $G$ must be $\sym$ because $G$ does
not fix the $0$. Assume that the stabilizer of the $0$ in $G$ is $\gl$.
Then the assumptions in Lemma \ref{affin} holds for $G$. So $G$ is a subgroup of $\aff$ containing $\gl$. Using that $\aff$ is a semidirect
product of $T$ and $\gl$, and $\aff$
is generated by $\gl$ and any translation (except the identity), and $G$ cannot be $\gl$ because $G$ does not fix the $0$.
We obtained that in this case $G$ must be the group $\aff$.

\end{proof}

\section{Orbits and reducts}

In Section~\ref{group} we have found the closed supergroups of
$\gl(\omega,2)$ in $\sym(\ff_2^\omega)$. Each of them corresponds to a first
order definable reduct of the vector space $\ff_2^\omega$. 

\begin{corollary} The vector space $\ff_2^\omega$ has 4 first order
definable reducts:
\begin{enumerate}
\item $\ff_2^\omega$, corresponding to the group $\gl(\omega,2)$,
\item The affine space corresponding to $\aff$,
\item The structure, with one unary relation, $\{0\}$ corresponding to
  $\sym_0(\ff_2^\omega)$,
\item The trivial (no relation) structure, corresponding to $\sym(\ff_2^\omega)$.
\end{enumerate}

\end{corollary}

Each reduct is homogeneous by looking at its automorphism group. On
the other hand, as we mentioned earlier, neither $\ff_2^\omega$, nor
the affine space are homogeneous on a finite language. As $\aff$ is 3
transitive, the first interesting problem is to describe the 4-types in the affine space.

\begin{lemma}\label{elo} The affine group $\aff$ has the following orbits on
the 4-tuples of the affine space:
\begin{enumerate}
\item $(a,a,a,a)$ for any $a\in \ff_2^\omega$
\item $(a,a,a,b)$ for any $a,b\in \ff_2^\omega$, where $a\neq b$ and
  all its cyclic 
  permutations. 
\item $(a,a,b,b)$ for any $a,b\in \ff_2^\omega$, where $a\neq b$ 
\item $(a,b,a,b)$ for any $a,b\in \ff_2^\omega$,  where $a\neq b$
\item $(a,b,b,a)$ for any $a,b\in \ff_2^\omega$, where $a\neq b$ 
\item  $(a,a,b,c)$ for any $a,b,c\in \ff_2^\omega$, where
  $|\{a,b,c\}|=3$, and all six permutations of this tuple.
\item\label{ind}  $(a,b,c,d)$ for any $a,b,c,d\in \ff_2^\omega$, where
  $|\{a,b,c,d\}|=4$, and $a+b+c+d\neq 0$.
\item\label{dep}  $(a,b,c,d)$ for any $a,b,c,d\in \ff_2^\omega$, where
  $|\{a,b,c,d\}|=4$, and $a+b+c+d= 0$. 
\end{enumerate}

\end{lemma}

\begin{proof}
The first six items follow from the 3-transitivity of $\aff$. For items
\eqref{ind} and \eqref{dep} let $a,b,c,d$ be distinct elements of the
affine space. 
Again, by the 3-transitivity of $\aff$ we may assume that $a,b,c$ are
linearly independent. If $d\notin\langle a,b,c   \rangle$, then
$(a,b,c,d)$ belongs to case \eqref{ind}. 
If $d\in\langle a,b,c   \rangle$, then as $a,b,c,d$ are distinct
either $d$ is the sum of two or the sum of three elements. In the
first case we may assume that $a+b=d$. Then $\langle
a,b\rangle=\{a,b,d,0\}$ and so $c\notin \langle a,b,d\rangle$. Then, by
Lemma~\ref{kislemma} $(a,b,c,d)$ belongs to case \eqref{ind}. 
In the second case $a+b+c=d$, or equivalently $a+b+c+d=0$, hence we
are in case \eqref{dep}.
\end{proof}

And now we present an example of a reduct of $\ff_q^\omega$ 
for $q>2$.

\begin{example}\label{pld1}
Let $H< \ff_q^*$, a subgroup of the multiplicative group of
$\ff_q$ and let $|H|=k$. Note that $H$ is cyclic and  $k|p-1$.
As $H$ is a subgroup of the
multiplicative group, every $h\in H$ acts by multiplication on
$\ff_q^\infty$. For convenience we shall write $h\cdot v$ instead of
$v^h$ to distinguish the action of $H$ on $\ff_q^\infty$.
Define the relation $\sim_H$ on
$\ff_q^\infty\setminus \{0\}$ in the 
following way:
{ for  } $a,b\in\ff_q^\omega \setminus \{0\}$ { let } $a\sim_H b$ 
{ if  there is an } $h\in H$ { such that } $h\cdot a=b   $. The relation
$\sim_H$  is an equivalence relation and every $\sim_H$ class is
contained in a 1-dimensional subspace of $ \ff_q^\omega$. Let
$\mathcal V$ denote the set of equivalence classes of $\sim_H$. 

Now, we define the subgroup  $G<\sym(\ff_q^\omega )$ as the set of permutations which preserve
the relation $\sim_H$. This group acts on $\mathcal V$ as the symmetric group, so it is not the
group $\gl(\omega, q)$, and preserves a nontrivial relation, so it can not be $\sym(\ff_q^\omega)$ either.  
\end{example}

Hence, finding the closed supergroups of $\gl(\omega, q)$ for $q>2$
will require different techniques.

\begin{problem}
Find the closed groups $G$ satisfying $\gl(\omega,q) \leq G \leq \sym
(\ff_q^\omega)$. Equivalently: find the fist order definable reducts
of the vector space $\ff_q^\omega$.
\end{problem}

\section{Vector space endowed with a symplectic bilinear product}

In this section we will investigate the closed groups containing the infinite dimensional counterpart of the finite
symplectic groups.

Fra\"iss\'e's theorem states that every homogeneous structure can be obtained as the Fra\"iss\'e-limit of its age (the class
of its finitely generated substructures) \cite{Frass}.
We would like to define the structure $\mathbb{F}_2^\omega(+,\cdot)$ as the Fra\"{i}ss\'{e}-limit of finite
dimensional vector spaces over $\mathbb{F}_2$ endowed with a symplectic bilinear form $\cdot$.
The problem with this definition is that if we enrich the structure of a vector space
with a symplectic bilinear form, then it is no longer a first order structure.
To get around this problem we can add the binary relations $P_i(x,y)$
to the vector spaces $\ff_2^n$ which will express $x\cdot y=i$ $(i=0,1)$.
In this case the automorphisms of this structure are exactly those vector space automorphisms
which preserve the symplectic bilinear product $\cdot$.

\begin{definition}\label{age}
Let $\mathcal{F}$ denote the class of finite dimensional vector spaces over $\mathbb{F}_2$
with binary relations $P_0,P_1$ for which the following axioms hold:
\begin{enumerate}
\item $\forall x,y(P_0(x,y)\leftrightarrow \neg P_1(x,y))$
\item For all $i,j\in \mathbb{F}_2$ the formula $\forall x,y,z (P_i(x,z)\wedge P_j(y,z)\rightarrow P_{i+j}(x+y,z))$.
\item For all $i,j\in \mathbb{F}_2$ the formula $\forall x,y,z (P_i(x,y)\wedge P_j(x,z)\rightarrow P_{i+j}(x,y+z))$.
\item $\forall x (P_0(x,x))$.
\end{enumerate}
\end{definition}

Note that non-degenerateness is not required.
In \cite{Frass} a Fra\"{i}ss\'{e}-class is defined as a class of finitely generated structures satisfying the three
properties described below. For every such class Fra\"iss\'e's theorem guarantees the existence of a countable homogeneous structure such that
its class of finitely generated substructures is precisely the given class. The three required properties:
\begin{itemize}
\item{(HP)} Hereditary property: if a structure $S$ is in $\mathcal{F}$ then every finitely generated substructure of $S$ is in
$\mathcal{F}$
\item{(JEP)} Joint embedding property: if $S_1$ and $S_2$ are two structures from $\mathcal{F}$ then there is a structure $S_3$
in $\mathcal{F}$ such that both $S_1$ and $S_2$ can be embedded into it.
\item{(AP)} Amalgamation property: if $S_1, S_2$ and $S_3$ are three structures from $\mathcal{F}$ and $\phi_1$ and $\phi_2$
are embeddings of $S_3$ into $S_1$ and $S_2$ respectively then there exist a structure $S_4$ in $\mathcal{F}$
and embeddings $\psi_1 : S_1 \rightarrow S_4$ and $\psi_2 : S_2 \rightarrow S_4$ such that the embeddings $\phi_1 \circ \psi_1$
and $\phi_2 \circ \psi_2$ are the same ($\phi_1 \circ \psi_1=\phi_2 \circ \psi_2$ embeds $S_3$ into $S_4$).
\end{itemize}
These three properties are satisfied by $\mathcal{F}$.

Now the definition of $\mathbb{F}_2^\omega(+,\cdot)$ is as follows.

\begin{definition}
Let $\mathcal{V}=\mathbb{F}_2^\omega(+,\cdot)$ be the Fra\"{i}ss\'{e}-limit of the class $\mathcal{F}$ defined in Definition \ref{age}.
We will define the function $\cdot:\mathcal{V}\times \mathcal{V} \rightarrow \ff_2$ as follows:
$$x\cdot y=i \text{ iff } P_i(x,y) \text{ holds}.$$
\end{definition}

Because of the axioms in Definition \ref{age}
the function $\cdot$ is well-defined and is a non-degenerate symplectic bilinear form on $\mathbb{F}_2^\omega.$
Moreover, the automorphism group of $\mathbb{F}_2^\omega(+,\cdot)$ as a first order structure is the
group of those vector space automorphisms which preserve the symplectic bilinear product $\cdot$. 

\begin{proposition}
The structure $\mathcal{V}=\mathbb{F}_2^\omega(+,\cdot)$ is homogeneous, and $\omega$-categorical.
\end{proposition}

\begin{proof}
The homogeneity is guaranteed by Fra\"iss\'e's theorem. The $\omega$-categoricity is equivalent to the oligomorphity of the
automorphism group \ref{ERNS}. The automorphism group will be oligomorphic because for every $n$ there are finitely
many (possibly $0$) isomorphism types of finitely generated substructures, and by homogeneity two isomorphic substructures lie in
the same orbit of the automorphism group.
\end{proof}

We will denote the automorphism group of the structure $\mathcal{V}=\mathbb{F}_2^\omega(+,\cdot)$ by $\aut(\mathbb{F}_2^\omega(+,\cdot))=\symp$

\begin{proposition}\label{generic}
If $a_1,a_2,\ldots a_n\in \mathcal{V}$ are linearly independent, then for all $i_1,i_2,\ldots i_n\in \mathbb{F}_2$
there exists an element $w\not\in \langle a_1,\dots a_n \rangle$ in $\mathcal{V}$ such that
$a_j\cdot w=i_j$ for all $j=1,2,\dots ,n$.
\end{proposition}

\begin{proof}
This is a special case of the extension property of homogeneous structures.
Let $b_1, b_2 \ldots b_{n+1}$ be a base of the vector space $\ff_2^{n+1}$. We will define a $\cdot$ bilinear form on
$\ff_2^{n+1}$ to do so is enough to define the values $b_j \cdot b_k$ for all possible pairs of base elements. Let
$b_j \cdot b_k$ be
\begin{itemize}
\item $0$ if $j=k$
\item $0$ if $j\neq k$, $j\leq n$, $k \leq n$ and $P_0(a_j,a_k)$ holds in $\mathcal{V}$
\item $1$ if $j\neq k$, $j\leq n$, $k \leq n$ and $P_1(a_j,a_k)$ holds in $\mathcal{V}$
\item $x$ if $j \leq n$, $k=n+1$ and $i_j=x$
\item $x$ if $k \leq n$, $j=n+1$ and $i_k=x$
\end{itemize}
We define the relations $P_0$ and $P_1$ on $\ff_2^{n+1}$ as usual, and this yields a structure from the class $\mathcal{F}$
 \ref{age}. This structure will be denoted by $F_1$. The structure $\mathcal{V}$ has a substructure isomorphic to $F_1$:
 this substructure will be denoted by $F_2$. Denote the substructure of $\mathcal{V}$ generated by the elements
 $a_1,a_2,\ldots a_n\in \mathcal{V}$ by $F_3$. Denote the substructure of $F_2$ (and thus of $\mathcal{V}$) generated
 by the elements $\phi(b_1), \phi(b_2) \ldots \phi(b_{n})$ where $\phi$ is an isomorphism from $F_1$ to $F_2$ by $F_4$.
 Then $F_2$ and $F_4$ are isomorphic so there is a $\psi$ automorphism of $\mathcal{V}$ extending the isomorphism between them.
 This $\psi$ maps $F_2$ onto a substructure of $\mathcal{V}$ containing $F_3$: the element $\psi(\phi(b_{n+1}))$ will be a good
 choice for $w$. 
\end{proof}

\begin{proposition}\label{autp_0}
$\symp=\aut(P_0)$.
\end{proposition}

\begin{proof}
The inclusion "$\subset$" is obvious. For the other containment, we have to show that the binary function $+$ and the
relation $P_1$ are first order definable from the relation $P_0$. The latter is obvious since
$P_1(x,y)\Leftrightarrow \neg P_0(x,y)$. We know that $\cdot$ is non-degenerate, hence $x=0$ holds if and only
if $\forall y(P_0(x,y))$ holds. This implies that 0 is definable from $P_0$. Now, we claim that for all
$x,y,z\in \mathcal{V}\setminus \{0\}$ the equality $x+y=z$ holds if and only if
$$z\neq x\wedge z\neq y\wedge \forall w((P_0(w,x) \wedge P_0(w,y))\rightarrow P_0(w,z))$$ holds.

At first, assume $x,y,z\in \mathcal{V}\setminus 0$ and $x+y=z$. Then $z$ is not equal to $ x$ or $y$
and if $w\cdot x=w\cdot y=0$,
then $w\cdot z=w\cdot x+w\cdot y=0$. Now assume that $x,y,z\in \mathcal{V}\setminus 0$ and $z$ is not equal to
$x$ or $y$ or $x+y$.
Then $z\not\in \langle x,y \rangle$, thus by Lemma \ref{generic} it follows that there exists a $w\in \mathcal{V}$
such that $w$ is orthogonal to $x,y$ but it is not orthogonal to $z$. Therefore
$\forall w((P_0(w,x) \wedge P_0(w,y))\rightarrow P_0(w,z))$ does not hold. 

Now, the relation $x+y=z$ can be defined in general as follows: 
\begin{multline*}
x+y=z\Leftrightarrow (x=y \wedge z=0) \vee (y=z \wedge x=0) \vee (z=x \wedge y=0) \vee \\
\vee (x\neq 0 \wedge y\neq 0 \wedge z\neq 0 \wedge z\neq x \wedge z\neq \\
y \wedge \forall w((P_0(w,x) \wedge P_0(w,y))\rightarrow P_0(w,z)).
\end{multline*}
As we have seen, the relation $x=0$ is definable from $P_0$, thus this gives us a first order definition of $+$
from the relation $P_0$.
\end{proof}
 
First we will deal with the groups fixing the $0$.
 
Let $\mathcal{B}$ denote a basis of the vector space $\mathcal{V}$.
Then the relation $P_0$ defines a graph $\mathcal{G}$ on the domain $\mathcal{B}$:
two elements of $\mathcal{B}$ will be connected with an edge if and only if
$P_1(x,y)$ holds (that is $P_0(x,y)$ does not hold).
It is easy to see that in this case $\mathcal{G}$ is an undirected graph without loops. 

\begin{proposition}
We can choose the basis $\mathcal{B}$ in such a way that the graph $\mathcal{G}$ defined as above will be isomorphic
to the random graph.
\end{proposition}

\begin{proof}
We will construct a basis with the given property using the back-and-forth method. Let $R$ denote the random graph 
and enumerate all of its vertices as $a_1,a_2,a_3 \ldots$. Enumerate the elements of $\mathcal{V}$ as $b_1,b_2,b_3 \ldots$. We will denote the graph
obtained from a subset $S$ of $\mathcal{V}$ by connecting elements $a,b \in S$ with an edge if and only if $P_1(a,b)$ holds by $\mathcal{G}(S)$

We will construct an $R_i$ sequence of finite subgraphs of $R$ and an $S_j$ sequence of finite linearly independent subsets of $\mathcal{V}$
such that
$$R_1 \subset R_2 \subset R_3 \mbox{ and } R=\cup R_i $$
$$S_1 \subset S_2 \subset S_3 \mbox{ and } S=\cup S_j \mbox{ is a basis in } \mathcal{V} $$
$$R_i \cong \mathcal{G}(S_i) \mbox{ for every } i $$
Let $R_1$ be an arbitrary vertex of $R$ and $S_1$ be an arbitrary nonzero element of $\mathcal{V}$. We will use recursion:
\begin{itemize}
\item If $i$ is even and $R_j$ and $S_j$ are already defined for all $j<i$. We choose $R_i$ as the subgraph of $R$ determined by the
vertices of $R_{i-1}$ and that vertex of $R \setminus R_{i-1}$ which has the least index in the series $a_1,a_2,a_3 \ldots$. Then we can choose
an element $b_l$ from $\mathcal{V}\setminus S_{i-1}$ such that $R_i \cong \mathcal{G}(S_{i-1}\cup\{b_l\})$ by Proposition \ref{generic}. Let $S_i$ be $S_{i-1}\cup\{b_l\}$.
\item If $i$ is odd and $R_j$ and $S_j$ are already defined for all $j<i$. We choose $S_i$ as $S_{i-1} \cup b_k$ where $b_k$ is that element
of $\mathcal{V}$ which has the least index in the sequence $b_1,b_2,b_3 \ldots$ amongst the elements linearly independent from $S_{i-1}$.
Then we can choose an $R_i \supset R_{i-1}$ subgraph of $R$ such that $R_i \cong \mathcal{G}(S_i)$ because of the well-known extension property of the random graph.
\end{itemize}
The subset $S$ of $\mathcal{V}$ will be a basis satisfying the requirements of the Lemma.
\end{proof}

Now, let us fix a basis $\mathcal{B}$ with the above property, and let us fix the corresponding graph $\mathcal{G}$ as well.
If $G$ is an arbitrary permutation group acting on the domain set of $\mathcal{V}$
then let $G^\mathcal{B}$ denote closure of the action of the group $G_\mathcal{B}$ on $\mathcal{B}$ in
the symmetric group $\sym(\mathcal{V})$. Here $G_\mathcal{B}$ denotes the setwise stabilizer of $\mathcal{B}$ in $G$.
Note that using this notation 
$\symp^\mathcal{B}=\aut(\mathcal{G})$.

In the following few lemmas we will show that if an element $g$ violates a certain type of relation on linearly independent elements of $\mathcal{V}$, then it can be realized in the group $\langle\symp,g\rangle^\mathcal{B}$.

\begin{lemma}\label{ext_1}
	Let $S$ be a finite subset of $\mathcal{B}$. Suppose we have a function $g\in \sym(\mathcal{V})$ such that $a^g$ is in $\mathcal{B}$ for all $a\in S$. Let $b\in \mathcal{B}$ be an arbitrary element. Then there exists an element $g'\in \langle g,\symp \rangle$ such that $g'$ agrees with $g$ on $S$ and $b^g$ is in $\mathcal{B}$.
\end{lemma}

\begin{proof}
If $b\in S$, then the statement of the lemma is trivial. Suppose that it is not the case. Then $b\not\in \langle S \rangle$,
	hence by the homogeneity of $\mathcal{V}$ the orbit of $b$
	in the pointwise stabilizer of $S$ in $\symp$ is infinite. In particular there
	exists an element $x$ in this orbit such that $x^g \not\in \langle S^g\rangle$ so the elements of $S^g\cup \{x^g\}$ are linearly independent. By the universal property of the graph $\mathcal{G}$
	there exists an $y \in \mathcal{B}$ such that $P_0(a^g, x^g)\Leftrightarrow P_0(a^g, y)$ for all
	$a\in S$. Then since both $S^g\cup \{x^g\}$ and $S^g\cup \{y^g\}$
	are linearly independent sets, it follows that there exists an automorphism $\gamma$ of $\mathcal{V}$ such that
	$a^{g\gamma}=a$ for all $a\in S$ and $x^{g\gamma}=y$. By the definition of $x$ we know that there exists an
	automorphism $\delta$ of $\mathcal{V}$ such that $a^{\delta}=a$ for all $a\in S$ and $b^{\delta}=x$. Then the
	permutation $g'=\delta g \gamma\in \langle g,\symp \rangle$ will satisfy our requirements.
\end{proof}

\begin{corollary}\label{ext_2}
	Let $S\subset T_1,T_2$ be finite subsets of $\mathcal{B}$. Suppose we have a function $g\in \sym(\mathcal{V})$ such that $a^g$ and $a^{g^{-1}}$ are in $\mathcal{B}$ for all $a\in S$. Then there exists an element $g'\in \langle g,\symp \rangle$ such that
\begin{itemize}
\item $g'$ and $g$ agrees on $S$,
\item $g'^{-1}$ and $g^{-1}$ agrees on $S$,
\item For all $a\in  T_1$ : $a^{g'}\in \mathcal{B}$,
\item For all $a\in T_2$ : $a^{g'^{-1}}\in \mathcal{B}$.
\end{itemize} 
\end{corollary}

\begin{proof}
	Let $n:=|T_1\setminus S|+|T_2\setminus S|$. In the proof we will use induction on the value of $n$. If $n=0$, then the statement is trivial. 
	Now, assume $n>0$. By switching to the inverse function if it is necessary, we can assume that $T_1\setminus S\neq \emptyset$. Then let $b\in T_1\setminus S$ an arbitrary element. By the induction hypotheses we know that there exists an element $g''\in \langle g,\symp \rangle$ such that \begin{itemize}
\item $g''$ and $g$ agrees on $S$,
\item $g''^{-1}$ and $g^{-1}$ agrees on $S$,
\item For all $a\in  T_1\setminus \{b\}$ : $a^{g'}\in \mathcal{B}$,
\item For all $a\in T_2$ : $a^{g'^{-1}}\in \mathcal{B}$.
\end{itemize} 
	Now, by applying Lemma \ref{ext_1} to the set $S'=T_1\setminus \{b\} \cup T_2^{g^{-1}}$ and the element $b$ we obtain a permutation $g'\in \langle g',\symp \rangle\subset \langle g,\symp \rangle$ such that $g'$ agrees with $g''$ on $S'$ and $b^{g'}\in \mathcal{B}$. It is easy to check that this $g'$ satisfies the conditions of the corollary.
\end{proof}

\begin{lemma}\label{ext}
Let $\Phi(x_1,x_2,\dots x_n)$ be an arbitrary quantifier-free formula in the language $L=\{P_0\}$,
and assume that for an element $g\in \sym(\mathcal{V})$ there exist $a_1,a_2,\dots, a_n\in \mathcal{V}$ such that 
\begin{itemize}
\item $a_1,a_2,\dots a_n$ are linearly independent,
\item $a_1^g,a_2^g,\dots a_n^g$ are linearly independent,
\item $\Phi(a_1,a_2,\dots a_n)$ is true,
\item $\Phi(a_1^g,a_2^g,\dots, a_n^g)$ is false.
\end{itemize}
Then there exists an $h\in \langle\symp,g\rangle^\mathcal{B}$ and $b_1,b_2,\dots ,b_n\in \mathcal{B}$
such that $\Phi(b_1,b_2,\dots, b_n)$ is true, but $\Phi(b_1^h,b_2^h,\dots ,b_n^h)$ is false.
\end{lemma}

\begin{proof}
	Let $G$ denote the following graph: the vertices of $G$ are $a_1,a_2, \dots a_n$
	and two vertices $a_i$ and $a_j$ are connected if and only if $P_0(a_i,a_j)$ is false.
	Then the graph $G$ is a finite undirected graph without loops,
	therefore it can be embedded into $\mathcal{G}$.
	Let $\psi$ be an arbitrary embedding of $G$ into $\mathcal{G}$. 
	Then $\Phi(a_1^\psi,a_2^\psi,\dots,a_n^\psi)$ is true because $\Phi$ is a quantifier-free formula.
	Since both $\{a_1,a_2,\dots a_n\}$ and $\{a_1^\psi,a_2^\psi,\dots, a_n^\psi\}$ are linearly independent
	sets it follows
	from the homogeneity that $\psi$ extends to an automorphism of $\mathcal{V}$.
	We will denote this automorphism by $\psi$ as well. Let $G'$ be the graph on the vertices
	$a_1^g,a_2^g,\dots ,a_n^g$ defined similarly, and let $\psi'$ be an embedding of $G'$ into $\mathcal{G}$.
	Then $\Phi(a_1^{g\psi'},a_2^{g\psi'},\dots,a_n^{g\psi'})$ is false and $\psi'$ extends to an automorphism of $\mathcal{V}$ as well.

	Let $b_1,b_2,b_3,\dots$ be an enumeration of $\mathcal{B}$ such that $b_i=a_i^\psi$ for $i=1,2,\dots n$.
	Then by using Corollary \ref{ext_2} we can define a sequence of functions $h_n,h_{n+1},h_{n+2},\dots \in \left< \symp,g \right> ^{\mathcal{B}}$ by recursion such that 
\begin{itemize}
\item $h_n$ and $\psi^{-1} g\psi'$ agrees on $\{b_1,b_2,\dots ,b_n\}$,
\item $h_n^{-1}$ and $\psi'^{-1} g^{-1}\psi$ agrees on $\{b_1,b_2,\dots ,b_n\}$,
\item $h_{k+1}$ and $h_k$ agrees on $\{b_1,b_2,\dots ,b_k\}$,
\item $h_{k+1}^{-1}$ and $h_k^{-1}$ agrees on $\{b_1,b_2,\dots, b_k\}$,
\item For all $i\leq k$ : $b_i^{h_k}\in \mathcal{B}$,
\item For all $i\leq k$ : $b_i^{h_k^{-1}}\in \mathcal{B}$.
\end{itemize} 	

	The sequences $h_k$ and $h_k^{-1}$ restricted to $\mathcal{B}$ are convergent, therefore there exists
	an $h\in \langle\symp,g\rangle^\mathcal{B}$ such that
	$h|_{\{b_1,b_2,\dots, b_k\}}=h_k|_{\{b_1,b_2,\dots, b_k\}}$ for all $k\geq n$. In
	particular $h|_{\{b_1,b_2,\dots, b_k\}}=h_k|_{\{b_1,b_2,\dots, b_k\}}=\psi^{-1} g\psi'|_{\{b_1,b_2,\dots, b_k\}}$.
	So the formula $\Phi(b_1^h,b_2^h,\dots, b_k^h)$ is false which finishes the proof of Lemma \ref{ext}.
\end{proof}

\begin{definition}\label{negyzetdef}
Let $\negyz (a,b,c,d)$ denote the following $4$-ary relation:
$$\negyz (a,b,c,d) \Leftrightarrow a\cdot b + b \cdot c + c \cdot d + d \cdot a =1 \mbox{ and }$$
$$a,b,c \mbox{ and } d \mbox{ are pairwise disjoint nonzero elements} $$
\end{definition}

\begin{lemma}\label{symptranz}
Let $\symp \leq G \leq (\sym (\mathcal{V}))_0$ be a closed group. If $G$ does not preserve $\negyz$
then $G^{\mathcal{B}}=\sym(\mathcal{B})$. Moreover, in this case for every $n$ the group $G$ acts transitively on the $n$-element
linearly independent sets.
\end{lemma}

\begin{proof}
If $G$ does not preserve $\negyz$ then there is a permutation $g \in G$ and pairwise disjoint elements
$a,b,c,d \in \mathcal{V}\setminus \{0\}$ such that $\negyz(a,b,c,d)$ holds but $\negyz(a^g,b^g,c^g,d^g)$ does not hold. We can assume that $a,b,c$ and $d$ are
linearly independent: there is an element $s \in \mathcal{V}\setminus \{0\}$ such that $s \notin \left< a,b,c,d\right>$ and $s^g \notin \left< a^g,b^g,c^g,d^g\right>$.
Observe that exactly zero or two holds from the formulas $\negyz(a,b,c,d),\negyz(a,b,s,d)$ and $\negyz(s,b,c,d)$, and similarly
exactly zero or two holds from the formulas $\negyz(a^g,b^g,c^g,d^g),\negyz(a^g,b^g,s^g,d^g)$ and $\negyz(s^g,b^g,c^g,d^g)$.
Using that exactly one formula from $\negyz(a,b,c,d)$ and $\negyz(a^g,b^g,c^g,d^g)$ holds
we can conclude that $g$
does not preserve the relation $\negyz$ on $(a,b,s,d)$ or on $(s,b,c,d)$. Similarly we can replace $b$ or $d$ by an element $r$ such that
$r$ and $r^g$ are linearly independent from the previous vectors. The elements $s,r$ and the remaining ones $x \in \{a,c\}$ and $y \in \{b,d\}$ will be linearly independent
and their images under $g$ will also be independent.

The relation $\negyz$ was defined by a quantifier-free formula \ref{negyzetdef} so we can apply Lemma \ref{ext} which yields that
the group $\aut(\mathcal{G}) \leq G^{\mathcal{B}} \leq \sym({\mathcal{G}})$ does not preserve the relation $\negyz$.
Using the classification obtained by Thomas
in \cite{thomas1} all closed groups containing $\aut{(\mathcal{G})}$ preserve $\negyz$ except the group $\sym(\mathcal{B})$
so $G^{\mathcal{B}}=\sym(\mathcal{B})$.

We can prove that for every $n$ the group $G$ acts transitively on the $n$-element
linearly independent sets by showing that any finite linearly independent sets can be mapped into $\mathcal{B}$ by an automorphism of
$\mathcal{V}$. Now let $S$ be a finite set of linearly independent elements in $\mathcal{V}$. Let us consider the graph
$G$ the vertices of which are the elements of $S$ and two vertices $a,b$ are connected if and only if $P_0(a,b)$ is false. Then by the
universality of $\mathcal{G}$ the graph $G$ embeds into $\mathcal{G}$. Let $\psi$ be an arbitrary embedding of $G$ into $\mathcal{G}$.
Then by the homogeneity of $\mathcal{V}$ $\psi$ can be extended to an automorphism of $\mathcal{V}$. This automorphism maps $S$ into $\mathcal{B}$.
\end{proof}

\begin{lemma}\label{symplin}
If a permutation $g\in (\sym(\mathcal{V}))_0$ preserves the relation $\negyz$ then it is linear.
\end{lemma}
\begin{proof}
First we define the $5$-ary relation $\otsz(a,b,c,d,e)$:
$$\otsz (a,b,c,d,e) \Leftrightarrow  \mbox{ the number of unordered pairs } \{x,y\} \mbox{ such that }$$
$$x\neq y,x,y \in \{a,b,c,d,e\},x \cdot y = 1  \mbox{ is odd }$$
$$\mbox{ and } a,b,c, d \mbox{ and } e \mbox{ are pairwise disjoint nonzero elements} $$

The relation $\otsz$ is first order definable from the relation $\negyz$: $\otsz(a,b,c,d,e)$ holds if and only if the number of true formulas amongst
$\negyz(a,b,c,d),\negyz(a,b,c,e),\negyz(a,b,d,e),\negyz(a,c,d,e)$ and $\negyz(b,c,d,e)$ is odd. If $a,b$ and $c$ are linearly independent elements then
the truth value of the formula $\otsz(a,b,c,a+b+c,x)$ will be the same regardless of the choice of $x$. This is true because exactly zero or two of
$a\cdot x, b\cdot x, c \cdot x$ and $(a+b+c) \cdot x$ can be $1$. On the other hand, if $a,b,c$ and $d$ are four linearly independent element then
by Proposition \ref{generic} we can choose $x$ and $y$ such that $\otsz(a,b,c,d,x)$ holds and $\otsz(a,b,c,d,y)$ does not hold.

Let $h \in (\sym(\mathcal{V}))_0$ be a non-linear permutation. We will show that $h$ does not preserve the relation $\negyz$. It suffices to show that
$h$ does not preserve $\otsz$ because $\otsz$ is first order definable from $\negyz$. We can assume there are elements
$a,b,a+b \in \mathcal{V} \setminus \{0\}$
such that $a^h+b^h \neq (a+b)^h$ (if this is not the case then we can change to work with $h^{-1}$ instead of $h$ because a permutation and its inverse
both preserve or do not preserve a given relation). There exist $x$ and $y$ such that
$\otsz(a,b,a+b,(a^h+b^h+(a+b)^h)^{h^{-1}},x)$ holds and $\otsz(a,b,a+b,(a^h+b^h+(a+b)^h)^{h^{-1}},y)$ does not hold. The truth value of
$\otsz(a^h,b^h,(a+b)^h,a^h+b^h+(a+b)^h,x^h)$ and $\otsz(a^h,b^h,(a+b)^h,a^h+b^h+(a+b)^h,y^h)$ must be the same. So $h$ can not preserve $\negyz$.

So every permutation preserving $\negyz$ must be linear.
\end{proof}

\begin{theorem}\label{sympsmall}
If $\symp \leq G \leq \sym(\mathcal{V})_0$ is a closed group then either $G=\symp$, $G=\gl$ or $G=\sym(\mathcal{V})_0$.
\end{theorem}
 
\begin{proof}
First suppose $\symp \lneqq G \leq \gl$. Then by Proposition \ref{autp_0} there is an element $g\in G$ which does not preserve the relation $P_0$.
Now, let $a,b\in \mathcal{V}\setminus 0$ elements such that $a\cdot b\neq a^g\cdot b^g$. We can assume that $a\cdot b=0$ and $a^g \cdot b^g=1$.
Then let us choose an element $c\in \mathcal{V} \setminus \langle a,b,\rangle$ such that $c\cdot a=c\cdot b=0$.
Then $a \cdot c = c \cdot b = b \cdot (a+c) = (a+c) \cdot a =0$ so $\negyz(a,c,b,a+c)$ does not hold. Furthermore
$a^g \cdot c^g + c^g \cdot b^g + b^g \cdot (a+c)^g + (a+c)^g \cdot a^g = a^g \cdot b^g =1$ therefore $g$ does not preserve the relation $\negyz$.
Then by Lemma \ref{symptranz} we know that for all $n$ the group $G$ acts transitively on $n$-element linearly independent sets. 
Now, let $S$ and $T$ finite dimensional subspaces of $\mathcal{V}$ of the same dimension, and let $\psi$ be a $S\rightarrow T$ linear isomorphism.
We need to show that there exists a linear automorphism $h\in G$ which extends $\psi$ because this extension property characterizes $\ff^\omega_2$.
Let $b_1,b_2,\dots,b_n$ be a basis of $S$. Then we know there exists an element $h\in G$ such that $b_i^h=b_i^\psi$ for all $i=1,2,\dots, n$
since $G$ acts transitively on $n$-element linearly independent sets. The transformation $h$ is linear as well, hence $h|_S=\psi$. So in this case
$\gl \leq G$.

Now suppose $G \lneqq \gl$. This means that $G$ does not preserve the relation $\negyz$ by Lemma \ref{symplin}, using Lemma \ref{symptranz} we get that
for all $n$ the group $G$ acts transitively on $n$-element linearly independent sets. We will prove that $G$ is $n$-transitive on
$\mathcal{V} \setminus \{0\}$ by induction modifying the proof of Theorem \ref{0fix}.

The group $G$ is $2$-transitive because $G$ acts transitively on two-element linearly independent sets.

We will show the $3$-transitivity.
The
group $G$ acts transitively on the 3-element independent
sets, hence it is enough to show that any three vectors $a,b,c\in \mathcal{V} \setminus \{ 0 \}$
can be mapped to an independent set.
Let $a,b,c\in \mathcal{V} \setminus \{ 0 \}$ and assume they cannot be mapped into an independent system. Then all $3$-tuples $(x,y,z)$ such that
$x+y=z$ are on the same orbit because the group $\symp$ has two orbits on such tuples: on the first orbit $P_0$ holds for all pairs of elements of the tuple,
and on the second orbit  $P_1$ holds for all pairs of elements of the tuple. Since $G$ is $2$-transitive $(a,b,c)$ can be mapped to $(x,y,z)$ where
$a \cdot b \neq x \cdot y$, using the assumption that $(a,b,c)$ cannot be mapped to an independent system $x+y=z$ so there is only one orbit on the linearly dependent
$3$-tuples in $G$.
The condition  $G \lneqq \gl$
implies that there are $a',b',c'\in \mathcal{V} \setminus \{ 0 \}$ and $g\in G$ such that
$a'+b'=c'$ and $a'^g+b'^g\neq c'^g $. Now, consider a map $h \in G$
mapping $a,b,c$ to  $a',b',c'$, respectively. This $h$ map exists because all $3$-tuples such that $a+b=c$ must lie on the same orbit. The map $hg$ maps
$a,b,c$ to an independent set. This contradicts our assumption that $(a,b,c)$ cannot be mapped to an independent set.

Now, we prove $n$-transitivity by induction. We show that every set of
$n+1$ vectors can be mapped to an independent set.
Let  $a_1, a_2 \ldots a_n, a_{n+1}
\in \mathcal{V} \setminus \{ 0 \}$ be dependent distinct elements. By the
$n$-transitivity we may assume that $a_n=a_1+ a_2 +\ldots+ a_{n-1}$ and $a_i \cdot a_j =0$ for all $1 \leq i < j \leq n$.
We can also assume that
there is an $h\in G$ such that  $\{a_i^h|i=1,2,\dots,n\}$ is
a linearly
independent set such that $a_i^h \cdot a_j^h =0$ for all $1 \leq i < j \leq n$.
If $a_{n+1} \notin \left<a_1,
a_2 \ldots a_n\right>$, then   by Lemma~\ref{kislemma} we are done. 
If $a_{n+1} \in \left<a_1,
a_2 \ldots a_n\right>$, then $a_{n+1}=\sum\limits_1^{n-1} \varepsilon_i
a_i $  where at least two, but not all $\varepsilon_i$ are equal to 1. 
Indeed, assume that there is a unique $i$ such that $\varepsilon_i=1$,
then $a_{n+1}=a_i $   would hold, and if all of them were equal to 1,
then
$a_{n+1}=\sum\limits_1^{n-1} 
a_i =a_n$ would hold contradicting that the vectors are distinct.
Let $\varepsilon_j=1$ and $\varepsilon_k=0$  for some $j,k<n$. Then
there is a map $g$ in the group $\symp$ flipping $a_j$ and $a_k$ and fixing every
$a_i$, where $i<n$ and $i\neq j,k$. Here we used that $a_i \cdot a_j =0$ for all $1 \leq i < j \leq n$ which also implies $a_i \cdot a_{n+1} =0$
for all $1 \leq i \leq n$ because $\cdot$ is bilinear.
Now,
$\{a_i^h|i=1,2,\dots,n\}=\{a_i^{gh}|i=1,2,\dots,n\}$ is an
independent set and $a_{n+1}^h\neq a_{n+1}^{gh} $. If any of the latter two
elements is not in $\left<a_1^h,
a_2^h, \ldots. a_n^h\right>$ then we are done by
Lemma~\ref{kislemma}. Otherwise we may assume that $
a_{n+1}^h=\sum\limits_1^n \xi_i a_i^h  $, where there is an $l$ such
that $\xi_l=0$. Now, $a_l^h\notin 
\left<a_i^h| 1\leq i\leq n+1, i\neq l\right>$
and we are done again, by Lemma~\ref{kislemma}.
\end{proof}



We have finished the classification of the closed groups containing $\symp$ which preserve the $0$. We will continue with the classification of groups
not fixing the $0$. Using Theorem \ref{sympsmall} and Theorem \ref{0nofix} we can restrict our attention to the groups where the stabilizer of the $0$
is exactly the group $\symp$. There is only one closed group which contains $\symp$ in addition to those already mentioned and it
can be obtained as the automorphism group of the relation defined below \ref{hszogaut}

\begin{lemma}\label{hszogaut}
Let $\hszog (a,b,c)$ denote the following ternary relation:
$$\hszog (a,b,c) \Leftrightarrow a\cdot b + b \cdot c + c \cdot a =1 \mbox{ and }$$
$$a,b \mbox{ and } c \mbox{ are pairwise disjoint elements} $$
The automorphism group of this relation can be obtained as a semidirect product:
$$\aut(\hszog)=T \rtimes \symp$$
Moreover, $\aut(\hszog)$ is a minimal supergroup: there is no $\symp \lneqq G \lneqq \aut(\hszog)$ closed group.
\end{lemma}

\begin{proof}
The group $\aut(\hszog)$ will be denoted by  $\symptr$. The following calculation shows that every translation preserves $\hszog$:
$t_x(a) \cdot t_x(b)+t_x(b) \cdot t_x(c)+ t_x(c) \cdot t_x(a)=(a+x)\cdot(b+x)+(b+x)\cdot(c+x)+(c+x)\cdot(a+x)=a\cdot b + b \cdot c + c \cdot a$.

$\symptr \leq \aff$ because $ \symp(\omega,2) \lneqq \symptr \leq \sym(\ff_2^\omega)$ and $\symptr_0= \symptr \cap (\sym(\mathcal{V}))_0=\symp$
so we can use Lemma \ref{affin}.

Using that $\aff = T \rtimes \gl$ and $\symp \leq \gl $ and $T \leq \symptr \leq \aff $ we can conclude that $\aut(\hszog)=T \rtimes \symp$.
Since $\symp$ acts transitively on $\mathcal{V} \setminus \{0\}$ every non-identical translation can be conjugated to any other non-identical translation
by an element of $\symp$. If $a,b \in \mathcal{V} \setminus \{0\}$ and $\phi \in \symp$ such that $\phi(a)=b$ then
$\phi t_b \phi^{-1}(x)=\phi ( \phi^{-1} (x) + b ) = x + \phi^{-1}(b)=t_a(x)$. So there is no $\symp \lneqq G \lneqq \symptr$ closed group.
\end{proof}

\includegraphics{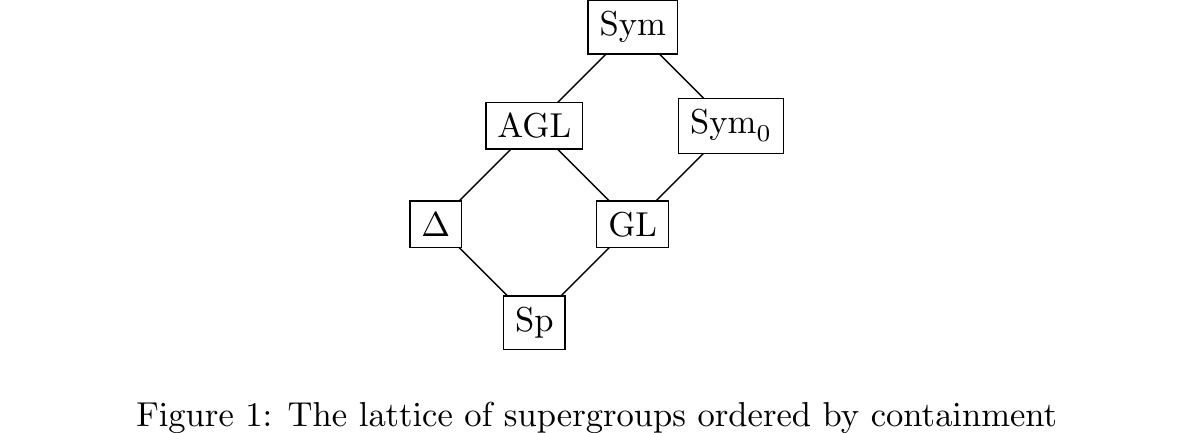}

\begin{lemma}\label{affin_symp}
Let us assume that $ \symp(\omega,2) \lneqq G \leq \sym(\ff_2^\omega)$, and $G_0=G\cap (\sym(\mathcal{V}))_0=\symp$.  Then
$G\leq \aff$.
\end{lemma}

\begin{proof}
Easy consequence of Lemma \ref{affin}.
\end{proof}

\begin{lemma}\label{twoorbits}
Let us assume that $ \symp(\omega,2) \lneqq G \leq \sym(\ff_2^\omega)$,
and $G_0=G\cap (\sym(\mathcal{V}))_0=\symp$.  Then $G$ has exactly 2 orbits on injective 3-tuples of $\mathcal{V}$ and both orbits contains linearly independent 3-tuples.
\end{lemma}

\begin{proof}
Let $(x,y,z)$ an arbitrary injective 3-tuple of $\mathcal{V}$. Since $\symp(\omega,2)$ is transitive on $\mathcal{V}\setminus 0$,
and $G$ does fix 0, it follows that $G$ is transitive. This means that $(x^g,y^g,z^g)=(0,a,b)$ for some $a,b\in \mathcal{V}\setminus 0$.
The group $\symp$ has exactly two orbits on injective 3-tuples of the form $(0,u,v)$. These orbit are distinguished by the product $u\cdot v$.
This implies that $G$ has \emph{at most} 2 orbits on injective 3-tuples. We also need to show that $G$ cannot be 3-transitive. Suppose for contradiction
that $G$ is 3-transitive. Then there exist distict elements $a,b\in \mathcal{V}\setminus 0$ and a permutation $g\in G$ such that $0^g=0$ and
$a\cdot b\neq a^g\cdot b^g$, but this contradicts the fact that $G_0=G\cap (\sym(\mathcal{V}))_0=\symp$.

	For the second statement of the Lemma it is enough to show that for any $a,b\in \mathcal{V}\setminus 0$ distinct elements the set $\{0,a,b\}$ can be mapped into a linearly independent set. Now, let $a\neq b$ be arbitrary elements in $\mathcal{V}$. Then there exist elements $a_1,a_2,a_3,a_4 \in \mathcal{V}\setminus 0$ such that $a_i\cdot a_j=a\cdot b$ for all $1\leq i<j\leq 4$. Let $g\in G$ a permutation which does not fix 0. Then $\dim \langle 0^g,a_1^g,a_2^g,a_3^g,a_4^g\rangle\geq 3$, thus the elements $0^g,a_i^g,a_j^g$ are linearly independent for some $1\leq i<j \leq 4$. Since $a_i\cdot a_j=a\cdot b$, there exists an $h\in \symp\subset G$ such that $a^h=a_i,b^h=a_j$. Then the permutation $hg$ maps the set $\{0,a,b\}$ into a linearly independent set. 
\end{proof}

\begin{lemma}\label{s4}
Let us assume that $ \symp(\omega,2) \lneqq G \leq \sym(\ff_2^\omega)$,
and $G_0=G\cap (\sym(\mathcal{V}))_0=\symp$. Let $(a,b,c,d)$ be an injective 4-tuple of $\mathcal{V}$ such that $a+b+c+d=0$. Then for any 3-element subset $\{s,t,u\}\subset \{a,b,c,d\}$ the tuples $(s,t,u)$ and $(a,b,c)$ are on the same orbit of $G$.
\end{lemma}

\begin{proof}
By Lemma \ref{affin_symp} and \ref{twoorbits} we get that the set of injective 4-tuples $(a,b,c,d)$ forming a two dimensional affine subspace of $\mathcal{V}$ is a union of two 4-orbits of $G$. This means that we only need to show the statement of the lemma for one single 4-tuple. 
	Let $a,b,c$ be linearly independent elements of $\mathcal{V}$ such that $a\cdot b=a\cdot c=b\cdot c=0$. Then the statement of the lemma is obvious.
\end{proof}

\begin{lemma}\label{hszoges}
Let us assume that $ \symp \lneqq G \leq \sym(\ff_2^\omega)$,
and $G_0=G\cap (\sym(\mathcal{V}))_0=\symp$. Then $G=\symptr$.
\end{lemma}

\begin{proof}
Let $G$ be a group satisfying the conditions of the lemma. Assume $G\neq \symptr$. Then by Lemma \ref{hszogaut} $G\lneqq \symptr$, so there is a permutation $g \in G$ which does not preserve the relation
$\hszog$. We claim that this can be realized by linearly independent element, i.e. there are elements $a,b$ and $c$ such that $\{a,b,c\}$ and $\{a^g,b^g,c^g\}$ are linearly independent sets, and $a\cdot b+b\cdot c+c\cdot a\neq a^g\cdot b^g+b^g\cdot c^g+c^g\cdot a^g$.

	By the definition of $g$ we know that there are elements $a,b,c\in \mathcal{V}$ such that $a\cdot b+b\cdot c+c\cdot a\neq a^g\cdot b^g+b^g\cdot c^g+c^g\cdot a^g$. By Lemma \ref{s4} we know that for all $u,v\in \mathcal{V}\setminus 0$ the tuples $(0,u,v)$ and $(u,v,u+v)$ lie on the same orbit. Note that $$\hszog(0,u,v)\Leftrightarrow \hszog(u,v,u+v)\Leftrightarrow P_1(u,v)$$ for any elements $u,v\in \mathcal{V}\setminus 0$. Using these observations we can assume that $a,b,c,a^g,b^g,c^g\neq 0$. Now, pick an element
$d$ such that both $d$ and $d^g$ are linearly independent from $\{a,b,c,a^g,b^g,c^g\}$. It is easy to check that for any injective $4$-tuple 
$(x,y,z,v)$ an even number
of the formulas $\hszog(x,y,z),\hszog(y,z,v),\hszog(z,v,x)$ and $\hszog(v,x,y)$ hold.  This implies that at least one of the tuples $(d,a,b),(d,b,c)$ and $(d,c,a)$ satisfies that
$g$ does not preserve the relation $\hszog$ on it (exactly one of the tuple and its image under $g$ is in $\hszog$) and these tuples contain linearly independent elements. This proves our claim.

	Now, we are ready to prove the statement of the lemma. For $i=0,1,2,3$ let $T_i$ denote the set of those linearly independent tuples $(a,b,c)$ tuples where exactly $i$ of the relations $P_1(a,b),P_1(b,c)$ and $P_1(c,a)$ hold. Then by Lemma \ref{s4} it follows that each set $T_i$ is contained in some 3-orbit of $G$. We would like to determine that which of these sets can be contained in the same orbit. At first, we show that $T_1$ and $T_3$ are on the same orbit. For this let us choose linearly independent elements $a,b,c$ in $\mathcal{V}$ such that $a\cdot b=a\cdot c=b\cdot c=1$. Then by Lemma \ref{s4} the tuples $(a,b,c)$ and $(a,b,a+b+c)$ lie on the same orbit, and it is easy to see that $(a,b,c)\in T_3$ and $(a,b,a+b+c)\in T_1$. 

	We have seen that there are elements $a,b$ and $c$ such that $\{a,b,c\}$ and $\{a^g,b^g,c^g\}$ are linearly independent sets, and $a\cdot b+b\cdot c+c\cdot a\neq a^g\cdot b^g+b^g\cdot c^g+c^g\cdot a^g$. This implies that either $T_0$ or $T_2$ is contained in the same orbit as $T_3$. We will deduce a contradiction in both cases. By Lemma \ref{twoorbits} we know that it is not possible that all $T_i$s are contained in the same orbit. 

	\paragraph{ Case 1. $T_0$ and $T_3$ are contained in the same 3-orbit of $G$} Then we know that there exists a 3-tuple $(a,b,c)\in T_3$ and a permutation $g\in G$ such that $(a^g,b^g,c^g)\in T_0$. Now, pick an element $d\in \mathcal{V}$ such that $d\not\in \langle a,b,c \rangle$, $d^g\not\in \langle a^g,b^g,c^g \rangle$, and $d\cdot a=0,d\cdot b=0, d\cdot c=1$. If $d^g\cdot a^g=d^g\cdot b^g=1$, then $(a,b,d)\in T_3$ and $(a^g,b^g,d^g)\in T_2$, which is impossible. Hence $d^g\cdot a^g=0$ or $d^g\cdot b^g$. By symmetry we can assume that the latter holds. Then $(b,c,d)\in T_2$ and $(b^g,c^g,d^g)\in T_1\cup T_0$. This is also impossible because we already know that $T_0,T_1$ and $T_3$ are contained in the same orbit.

	\paragraph{ Case 2. $T_2$ and $T_3$ are contained in the same 3-orbit of $G$}  Let $\{a,b\}$ and $\{c,d\}$ be two linearly independent sets such that $a\cdot b=0$ and $c\cdot d=1$. The group $G$ is transitive and the $G_0$ is transitive on $\mathcal{V}$, therefore $G$ is 2-transitive as well. In particular, there exists a permutation $g\in G$ such that $a^g=c$ and $b^g=d$. Now, pick an element $d\in \mathcal{V}$ such that $d\not\in \langle a,b\rangle$, $d^g\not\in \langle c,d \rangle$ and $d\cdot a=\cdot b=0$. Then $(a,b,d)\in T_0$ and $(a^g,b^g,d^g)$. This is again a contradiction since we already know that $T_1,T_2$ and $T_3$ are contained in the same orbit.

	We arrived to a contradiction in all cases, which proves the statement of the lemma.
\end{proof}

\begin{theorem}
The closed supergroups containing $\symp$ are exactly the following groups:
\begin{enumerate}
\item The group $\symp$
\item The group $\symptr$
\item The group $\gl$
\item The group $\aff$
\item The group $\symnull$
\item The group $\sym$
\end{enumerate}
\end{theorem}

\begin{proof}
The groups fixing the $0$ were described in Theorem \ref{sympsmall}: this groups are exactly $\symp, \gl$ and $\symnull$.
The groups not fixing the $0$ can be classified according to the stabilizer
of the $0$ in them. The groups where this stabilizer contains $\gl$ were described in in Theorem \ref{0nofix}:
this groups are exactly $\aff$ and $\sym$. The only group not fixing the $0$ where the stabilizer of the $0$ is $\symp$ is $\symptr$: this was proved in Lemma \ref{hszoges}.
\end{proof}

\end{document}